\documentclass[11pt]{amsart}

\usepackage[british]{babel}

\usepackage[latin1]{inputenc}
\usepackage{tabularx}
\usepackage{amscd}
\usepackage{times}
\usepackage[T1]{fontenc}

\newcommand{\PP}{\mathbb{P}}
\newcommand{\Gr}{\mathbb{G}}

\newcommand{\ZZ}{\mathbb{Z}}

\usepackage[titletoc,toc,title]{appendix}
\usepackage{amsfonts}
\usepackage{amsmath}
\usepackage{amsthm}
\usepackage{euscript}
\usepackage{amssymb}
\usepackage{enumerate}
\usepackage{hyperref}
\usepackage{mathrsfs}
\usepackage{xypic}
\usepackage{stmaryrd}
\xyoption{all}

\theoremstyle{plain}
\newtheorem{thm}{Theorem}[section]
\newtheorem{cor}[thm]{Corollary}

\newtheorem{prop}[thm]{Proposition}
\theoremstyle{definition}

\theoremstyle{remark}
\newtheorem{rem}{Remark}[section]

\numberwithin{equation}{section}

\newcommand{\C}{\mathbb C}

\newcommand{\TT}{\mathbb{T}}

\newcommand{\NN}{\mathcal{N}}

\SelectTips {cm}{}

\begin{document}

\title[Osculating behavior of Kummer surface in $\PP^5$]{Osculating behavior of  Kummer surface in $\PP^5$}

\author
{Emilia Mezzetti}
\address{Dipartimento di Matematica e Geoscienze,
  Universit\`a degli studi di Trieste, Via Valerio 12/1,
34127 Trieste -- Italy\\
\url{mezzette@units.it}
}
\thanks{Emilia Mezzetti is member of INdAM - GNSAGA and is supported by PRIN
``Geometria delle variet\`a algebriche'' and by
FRA, Fondi di Ricerca di Ateneo, Universit\`a di Trieste}

\subjclass[2010]{14J25, 14N05, 32J25, 53A20}
 \keywords{Kummer surface, osculating space, Laplace equation, second fundamental form}

\begin{abstract}
In an article of 1967 \cite{E67} W. Edge gave a description of some beautiful geometric properties of the Kummer surface complete intersection of three quadrics in $\PP^5$. Working on it, R. Dye proved that all its osculating spaces  have dimension  less than the expected $5$ (\cite{D82}, \cite{D92}).  Here we discuss these results, also at the light of some recent result about varieties with hypo-osculating behaviour. 
\end{abstract}
\maketitle

\section{Introduction}

The Kummer surface is a classical topic, whose  rich  and beautiful geometry continues to attract the attention of mathematicians. Introduced by Kummer in 1864 \cite{K}, it was extensively studied by Hudson in his book \cite{H}, among many others. In more recent times  we only mention \cite{GS}, where the modern point of view is well explained.

Actually the term \lq\lq\thinspace Kummer surface'' is now used to denote a more general object than a quartic surface in $\PP^3$ with sixteen nodes, as originally introduced. Let $A$ be an Abelian surface, let $\iota$ 
 be the involution $\iota : A \rightarrow A$, $a \rightarrow -a$. The quotient surface with respect to this involution  $A/\iota$ has
sixteen singular points of type $A1$, which are the images, under the quotient map,
of the sixteen points  of order two in $A$. The Kummer surface of $A$, $Km(A)$, is the  desingularization of $A/\iota$, obtained by blowing up these sixteen points. This gives rise to sixteen $(-2)$-curves.
The Picard number $\rho(Km(A))$  is equal to $\rho(A)+16$, so $17\leq \rho(Km(A))\leq 20$, and $\rho(Km(A))=17$ is the general case. 

All Kummer  are $K3$ surfaces, and  conversely Nikulin has characterized the Kummer as the $K3$ surfaces  containing sixteen disjoint rational curves \cite{N}. Among the projective  $K3$ surfaces, there are the smooth complete intersections of three quadrics in $\PP^5$, i.e. the surfaces that are base locus of a general net of quadrics $\NN$. In \cite{E67} Edge considers the smooth complete intersections of quadrics in $\PP^5$ that are Kummer. The property of being Kummer results to be equivalent to the conditions that the quadrics of the net $\NN$ have a common self-polar symplex $\Sigma$ and 
of containing a line in general position with respect to $\Sigma$.

In terms of equations, choosing the common self-polar symplex as reference symplex, one of the quadrics of $\NN$ can be expressed in canonical form by $x_0^2+x_1^2+\dots+x_5^2=0$; then the equations of the two other quadrics can be  put in the form $a_0x_0^2+a_1x_1^2+\dots+a_5x_5^2=0$, $a_0^2x_0^2+a_1^2x_1^2+\dots+a_5^2x_5^2=0$, where $a_0, a_1, \ldots,a_5$ are distinct.

 It is known (see for instance \cite{GS}) that the Kummer surface $X$ obtained in this way is $Km(A)$, where $A=\mathcal J(C)$, the jacobian of the plane quintic $C$ of affine equation $y^2=\prod_{i=0}^5 (x-a_i)$. In other words $C$ is a double covering of $\PP^1$ ramified at six points. $X$ contains $32$ lines and the N\' eron--Severi group $NS(X)$ is generated by $17$ of these lines, explicitly described in \cite{GS}. Projecting $X$ in $\PP^3$ from one of the lines contained in it, one obtains the Weddle surface, with six nodes,  birational to the classical quartic Kummer surface. In Hudson's book, the Weddle surface is constructed as image of a Kummer quartic $K$ by the rational map associated to the linear system  of sextic curves passing through ten nodes (\cite{H}, pag 166--167). A table states a dictionary between nodes and particular curves on the two surfaces.

We mention that $X$ has also another nice interpretation: it is the singular surface of a quadric congruence of lines in $\PP^3$, as explained for instance in \cite{GH78} and \cite{Do}.

Edge describes various geometric aspects of the surface $X$. Among many other facts, he proves that a complete intersection of three quadrics, having a common self-polar symplex $\Sigma$, contains one line in general position with respect to $\Sigma$ if and only if it contains $16$ lines. In this case, fixed a general point $y$ on $X$, there is precisely one quadric $Q_y$ in the net $\NN$ containing the tangent plane $\TT_yX$ to $X$ at $y$. The other quadrics of $\NN$ are shown to cut on $\TT_yX$ pairs of lines of an involution, that can be explicitly described. The two fixed lines of this involution, called asymptotic by Edge, are obtained from quadrics meeting $\TT_yX$ at a line with multiplicity two; they coincide if and only if $y$ belongs to one of the $32$ lines of $X$, which in this case  results to be the only fixed line of the involution. Edge proves (\cite{E67}, Sections 11-12) that in the birational map mentioned above, the asymptotic lines on $X$ correspond to the asymptotic lines, in the usual sense, of the Kummer surface $K$.

In \cite{D82} Dye extends the description given by Edge in the following way. He considers the tangent space to the quadric $Q_y$ at the point $y$, and proves that it coincides with $\TT^{(2)}_yX$, the (embedded) osculating space of $X$ at $y$. In the subsequent paper \cite{D92} he shows moreover that if a complete intersection of three quadrics in $\PP^5$, with a common self-polar symplex $\Sigma$, has one osculating space, at a point not on a face of $\Sigma$, of dimension less than $5$, which is the expected dimension, then it is a Kummer surface, so all osculating spaces have dimension less than  $5$. The dimension of the osculating spaces at points of $X$ lying on faces of $\Sigma$ further decreases, as it is also computed by Dye.  

Before continuing, let us briefly recall the following interpretation of the osculating spaces of a quasi-projective variety $X\subset \PP^N$ of dimension $n$.  
Let $x\in X$ be a smooth point. We can choose a system of affine
coordinates around $x$ and a local parametrization of $X$ of the
form $\phi(t_1,...,t_n)$ where $x=\phi(0,...,0)$ and the $N$
components of $\phi$ are formal power series.

One defines the $s$th osculating (vector) space
$T_x^{(s)}X$ to be the span of all
partial derivatives of $\phi$ of order
$\leq s$ at $(0,...,0)$. 
The expected dimension of $T_x^{(s)}X$
is ${{n+s}\choose {s}}-1$, but
in general $\dim T_x^{(s)}X\leq {{n+s}\choose {s}}-1$;
if strict inequality
holds for all smooth points of $X$, and $\dim
T_x^{(s)}X={{n+s}\choose {s}}-1-\delta$ for general $x$,
then $X$ is said to satisfy $\delta$ Laplace equations of order $s$.
Indeed, in this case the partials of order $s$ of $\phi$
are linearly dependent, which gives $\delta$ differential equations
of order $s$ which are satisfied by the components of $\phi$.
Of course for $s=1$, since $x$ is smooth, we get the tangent space to $X$ at $x$. If $s=2$, the second osculating space is often simply called the osculating space.
One can naturally consider also the embedded $s$th osculating space $\TT_x^{(s)}X$. If $X$ is a rational variety, parametrized by $N+1$ forms $F_0, \ldots,F_N$ of degree $d$ in $n+1$ variables, then from Euler's formula it follows that $\TT_x^{(s)}X$ is generated by the $s$th partial derivatives of the point $(F_0, \ldots,F_N)$ at $x$.
Varieties satisfying Laplace equations were classically studied in the realm of projective-differential geometry, where a foundational article is due to Corrado Segre \cite{S}, who adopted in greater generality the approach of Darboux (\cite{Da}). In more recent times, this point of view has been developed by Griffiths and Harris in \cite {GH79}. See also \cite{Sh} in the case of surfaces.

Therefore, the Kummer surface in $\PP^5$ satisfies one Laplace equation of order $2$. This result, which does not seem  to be very well known, is interesting because recently  there has been a renewed interest in  surfaces of $\PP^N$, $N\geq 5$, satisfying at least one Laplace equation of any order, and because there are few examples  known so far. 

The larger class of examples are the  surfaces uniruled by lines, i.e. such that through any general point passes a line contained  in the surface. In this case it is possible to find a parametrization such that one of the parameters appears at most at degree one, hence at least one second partial derivative vanishes  identically.  Apart from these and the Kummer surface, all known examples are rational surfaces. If the parametrization is given by cubic polynomials, the complete classification has been given by Eugenio Togliatti (\cite{T1}, \cite{T2}). For toric surfaces, partial results have been given by Perkinson \cite{P}. In \cite{MMO} a new point of view has been introduced, linking via apolarity rational varieties satisfying at least one Laplace equation with artinian ideals of  polynomial rings failing the Weak Lefschetz Property. This has led to the extension of the classification of Togliatti to rational varieties of any dimension parametrized by cubics (\cite{MMO}, \cite{MRM}), and to several new examples in larger degree (\cite{DIV}, \cite{MM1}, \cite{MM2} ). Surfaces satisfying a Laplace equation of order two have also been considered  in \cite{CM1}, \cite{CM2} by Chiantini and Markwig, who call them triple-point defective surfaces. They show that these surfaces, if they are linearly normal and of degree large enough, are necessarily ruled.

Another reason why the example of the Kummer surface seems interesting to us is because of the comments of Dye in both the quoted articles. In \cite{D82} Dye writes that Corrado Segre had claimed that, if the osculating space at a general point of a smooth surface of $\PP^N$, $N\geq 5$, is a $\PP^4$, then the surface contains two infinite families of lines. Dye also quotes a similar assertion by Room (\cite{S1}, \cite{R}). In \cite{D92} again Dye writes that Kummer's example contrasts the assertions of Segre and Room.

Aim of this article is to describe the interesting osculatory behaviour of Kummer surface,  and to explain the apparent contrast between this example  and claims of Segre. To this end, in  \S2 we will reproduce the results of Edge and Dye to determine the osculating space at a general point of the surface. Then in \S 3 we will explain how the theory of osculating spaces of a projective surface $X$ is related to its second fundamental form, and to the existence of conjugate directions at the points of $X$. From this it will result an interpretation of  the claims of Corrado Segre different from the one given by Dye, which is not in contrast with Kummer example. It will also offer a new simple way of computing the dimensions of its osculating spaces, only based on Edge's article.
\medskip

{\bf Notation.} We will always work over the complex field.
\medskip

{\bf Aknowledgements.} I had reflected on the osculating behavior of surfaces since many years, and some of the ideas exposed here were discussed in the graduation theses of my former students Orsola Tommasi and Marco Brazzach.  I'm grateful to Ivan Cheltsov who gave me the opportunity to write this article.

 
\section{Geometric properties of Kummer surface according to  Edge and Dye}


In this section we collect some basic facts about the geometry of the Kummer surfaces, from \cite{E67}, \cite{D82}, and \cite{D92}.

Let $X\subset \PP^5$ be a smooth surface complete intersection of three quadrics. Let $\NN$ be the corresponding net of quadrics. Assume that these quadrics have a common self-polar symplex $\Sigma$ and that $\Sigma$ is chosen as reference symplex. This means that $\NN$ is generated by three independent quadrics whose equations have the form
\begin{equation}\label{1}
\sum_{i=0}^5 \alpha_ix_i^2=0, \ \ \sum_{i=0}^5 \beta_ix_i^2=0, \ \ \sum_{i=0}^5 \gamma_ix_i^2=0. 
\end{equation}

In \cite{E67}, Section 3,  it is proved:
\begin{thm}
Let $X$ be a smooth surface  complete intersection of the quadrics in (\ref{1}). $X$ contains a line $l$ not meeting any plane face of $\Sigma$ if and only if there exist  $a_0, ..., a_5\in \C$, with $a_i\neq a_j$ if $i\neq j$, such that the net is generated by the quadrics 
 \begin{equation}\label{2}
Q_0: \sum_{i=0}^5 x_i^2=0, \ \ Q_1: \sum_{i=0}^5 a_ix_i^2=0, \ \ Q_2: \sum_{i=0}^5 a_i^2x_i^2=0. 
\end{equation}
In this case $X$ contains exactly $32$ lines, transformed of $l$ under the action of the elementary group of order $32$ generated by the projective transformations $h_i$, $i=0,...,5$, where $h_i$ replaces $x_i$ with $-x_i$, while leaving the other five coordinates unaltered.
\end{thm}

In the following of this section $X$ will be the surface complete interesection of the quadrics $Q_0, Q_1, Q_2$ of (\ref{2}).  We can consider the system of equations $Q_0=Q_1=Q_2=0$ as a linear homogeneous system in the unknown $x_i^2$.  Since the rank of the system is $3$ by construction, if $y=[y_0,...,y_5]\in X$, then there exist  constants $\xi, \eta, \zeta\in\C $  such that 
\begin{equation}\label{3}
y_i=\xi+\eta a_i + \zeta a_i^2,\end{equation} for $i=0,...,5$. To take advantage of the properties of the symmetric functions in $a_0,...,a_5$, Edge slightly modifies equations (\ref{3}) and writes the following equations
 \begin{equation}\label{4}
y_i f'(a_i)=\xi+\eta a_i + \zeta a_i^2,\end{equation} 
where $f(t)=\prod_{i=0}^5 (t-a_i)$ is the monic polynomial of degree $6$ having $a_0,...,a_5$ as roots.

This allows him to prove in Sections 5--6 of \cite{E67} the following result.
\begin{thm}\label{quadrics}
If $y\in X$ is given by (\ref{4}), then the tangent plane $\TT_yX$ has the following parametric equations in the parameters $p,q,r$: 
\begin{equation}\label{5}
x_i=(p+qa_i+ra_i^2)y_i/(\xi+\eta a_i+\zeta a_i^2).
\end{equation}
$\TT_yX$ is contained in the quadric $\xi Q_0+\eta Q_1+\zeta Q_2=0$  of the net $\NN$. The other quadrics of $\NN$ intersect $\TT_yX$ in a pair of lines passing through $y$. Moreover, using $p,q,r$ as homogenous coordinates in $\TT_yX$, the pairs of lines arising in this way are conjugate with respect to the conic $\Gamma(y)$ of equation $q^2-4rp=0$. 
\end{thm}

\begin{rem}\label{involution}  From Theorem \ref{quadrics} it follows that the quadrics of $\NN$ not containing $\TT_yX$ cut on it the pairs of lines of an involution in the pencil of lines of centre $y$. The fixed lines of this involution are the tangent lines to $\Gamma(y)$ through $y$. Each of them is the double intersection with $\TT_yX$  of a pencil of quadrics contained in the net $\NN$.
Moreover, the involution has a fixed line if and only if $y$ belongs to a line contained in $X$.
\end{rem}

The above description has been further extended by Dye.  He first determines for any point of $X$ the quadrics of the net $\NN$ containing the tangent plane $\TT_yX$.

\begin{thm} (\cite{D82}, Theorem 1)
Let $X$ be a smooth surface  complete intersection of the quadrics in (\ref{1}). Let $y\in X$, and let $T$ denote the tangent plane to $X$ at $y$. Let $E_0, ... E_5$ denote the vertices of the symplex $\Sigma$. Then:
\begin{enumerate}
\item if $y$ is not a point of intersection of two lines contained in $X$, then there is a unique quadric of $\NN$ containing $T$;
\item if $y$ is the point of intersection of two lines contained in $X$, then $y$ lies in one face of $\Sigma$. If it is the face opposite to $E_i$, then the quadrics of $\NN$ containing $T$ are the cones of $\NN$ through $E_i$; moreover  $E_i$ belongs to the vertex of any such cone;
\item in the pencil of cones of $\NN$ through $E_i$, there is exactly one cone containing no tangent plane to $X$ other than those at the 16 points of the face of $\Sigma$ opposite to $E_i$, that are intersections of pairs of lines on $X$.
\end{enumerate}
\end{thm}

So for any point $y$ in $X$ there is either a unique quadric, or a distinguished quadric of $\NN$ containing $\TT_yX$. This is called by Dye the osculating quadric of $X$ at $y$. Then he proves that the tangent space at $y$ to  the osculating quadric meets $X$ along a curve having a triple point at $y$. This allows him to compute the dimension of the osculating space $\TT^{(2)}_yX$ for any $y\in X$. 

In the following theorem we collect the results of Theorems 3 and 4 of  \cite{D82}. 

\begin{thm}
Let $X$ be a smooth surface  complete intersection of the quadrics in (\ref{1}). Let $T$ denote the tangent plane to $X$ at $y\in X$. Then 
\begin{enumerate}
\item a hyperplane $H$ cuts $X$ in a curve having $y$ as point of multiplicity at least $3$ if and only if H is the tangent hyperplane at $y$ to a quadric of the net $\NN$ containing $T$;
\item if $y$  is not a point of intersection of two lines contained in $X$, $\dim \TT_y^{(2)}X=4$;
\item if $y$  is the point of intersection of two lines contained in $X$, then $\TT_y^{(2)}X$ is the intersection of the tangent hyperplanes at $y$ to the quadric cones of $\NN$ containing $T$, and $\dim \TT_y^{(2)}X=3$.
\end{enumerate}
\end{thm}

The arguments used in the proof are geometric of syntetic type. In the successive article \cite{D92}, Dye gave another, purely analytic, proof using an explicit local parametrization of $X$. This second method applies also to compute the osculating spaces of all orders in the more general case of a complete intersection of $n-r$ quadrics in $\PP^n$, with a common self-polar symplex and equations generalizing (\ref{2}), for $r\geq 1$.
Finally we remark that the equations of the osculating quadric and of the osculating space at a general point $y\in X$ have respectively the following simple forms:
\begin{equation*} \sum_{i=0}^5 f'(a_i)y_i^2x_i^2=0, \ \ \  \sum_{i=0}^5 f'(a_i)y_i^3x_i=0. \end{equation*}


\section{C. Segre point of view and Dye's remarks}

We begin this section recalling the relationship among the osculating spaces, the Gauss map and the projective second fundamental form of a projective  surface. We will follow the approach in  \cite{GH79}), as presented in \cite{MT} (see also \cite{H}).

Let $X\subset \PP^N$ be a projective surface. The Gauss map of $X$ is the rational map from $X$ to the Grassmannian of planes $\gamma: X \dashrightarrow \Gr(2,N)$ such that, for a smooth point $x\in X$,  $\gamma(x)$ is the embedded tangent space $\TT_xX$. Its differential at a smooth point $x$
$$\gamma_{*x}: T_xX \rightarrow T_{\TT_xX}\Gr(2,N)=\hbox{Hom}(\hat\TT _xX,V/\hat\TT _xX)$$
 associates to the tangent vector at $x$ to a smooth curve $C\subset X$ the tangent vector at the point $\TT_xX$ to the curve $\gamma(C)$. Here we have denoted by $V$ the $\C$-vector space such that $\PP(V)=\PP^N$, and $\hat\TT_xX$ the cone over $\TT_xX$ in $V$. Note that $\hat x$ is contained in the kernel of $\gamma_{*x}(v))$ for any $v\in T_xX$, therefore $\gamma_{*x}$ induces a linear map 
$$\gamma'_{*x}: T_xX \rightarrow \hbox{Hom}(\hat\TT _xX/\hat x,V/\hat\TT _xX)=\hbox{Hom}(T_xX, V/\hat\TT_xX).$$  One checks that $(\gamma'_{*x}(v))(w)=(\gamma'_{*x}(w))(v)$, therefore $\gamma'_{*x}$ allows to define a bilinear map from $T_xX$ to the normal space to $X$ at $x$, $N_xX=V/\hat\TT_xX$, or equivalently a quadratic map  from the symmetric product $\hbox{Sym}^{(2)}T
_xX$ to  $N_xX$.  They are both called  projective second fundamental form of $X$ at $x$. 

The second fundamental form can also be interpreted as a linear system of quadrics $\mid II_x\mid$ in $\PP(T_xX)$. Indeed  the dual of $II_x: \hbox{Sym}^{(2)}T_xX\rightarrow N_xX$ is a map from $(N_xX)\check{}$, i.e. the space of tangent hyperplanes to $X$ at $x$, to  the dual of $\hbox{Sym}^{(2)}T_xX$, i.e. the space of  quadratic forms on the tangent space. It associates to a tangent hyperplane $H$ the projectivized tangent cone to the intersection $H\cap X$. The linear system of quadrics $\mid II_x\mid$  is the image of this map in $\PP(T_xX)$. A tangent hyperplane  is in the kernel if and only if it contains the osculating space to $X$ at $x$. Therefore there is the following important relation:
$$\dim \mid II_x\mid=\dim \TT^{(2)}_xX-\dim X-1= \dim \TT^{(2)}_xX-3.$$

We then recall the notion of {\it conjugate directions} for the second fundamental form. Let $v,w\in T_xX$ be tangent vectors. They represent conjugate directions at $x$ if $II_x(v,w)=0$, i.e. the points $[v],[w]\in \PP(T_xX)$ are conjugate with respect to all quadrics in $\mid II_x\mid$. If there is a self-conjugate vector, its direction is called  asymptotic  at $x$. 

Let us now examine all possible cases for a surface, keeping in mind that $\PP(T_xX)$ is a $\PP^1$. 

\begin{prop}\label{secondff} Let $X\subset \PP^N$ be a surface. Assume that $x$ is a general point of $X$. Then one of the following happens: 
\begin{enumerate}
\item $\mid II_x\mid=\emptyset$, if and only if $\TT_x^{(2)}X=\TT_xX$, if and only if $X$ is a plane;

\item $\dim \mid II_x\mid=0$, i.e. $\mid II_x\mid$ is a unique quadric, if and only if $\dim T_x^{(2)}=3$. In this case for any tangent direction there is a conjugate direction. If this happens, either $N=3$ or $X$ is a ruled developable surface;

\item $\dim \mid II_x\mid=1$, i.e. $\mid II_x\mid$ is a pencil of quadrics, if and only if $\dim T_x^{(2)}=4$.  These surfaces, when $N\geq 5$, were classically called $\Phi$ surfaces. In this case at a general point of $X$ there is exactly either one pair of conjugate directions, or one asymptotic direction.  The conjugate directions result to be the fixed points of the involution in $\PP(T_xX)$, whose pairs are the tangents to the singular hyperplane sections, i.e. the tangents to the cuspidal tangent hyperplane sections.

\item $\dim \mid II_x\mid=2$. This is the general case when $N\geq 5$. There are no conjugate directions.
\end{enumerate}
\end{prop}

We can now state the following theorem.

\begin{thm}
 The Kummer surface $X\subset \PP^5$ of equations (\ref{2}) is a $\Phi$ surface not covered by lines. Moreover the involution in $\PP(T_xX)$ considered in Remark \ref{quadrics} coincides with the one appearing in Proposition \ref{secondff} (3). 
\end{thm}
\begin{proof} The first assertion follows from the discussion in \S 2 and from Proposition \ref{secondff}. The second assertion follows because
 the pairs of lines cut on the tangent plane $\TT_xX$ by the non-osculating quadrics containing $X$ clearly coincide with the pairs of lines of the tangent cone of the hyperplane sections having $x$ as double point. Indeed these hyperplane sections are cut by the tangent hyperplanes at $x$ to the non-osculating quadrics containing $X$.
\end{proof}

\begin{cor} Let $y\in X$ and $T=\TT_yX$ be the embedded tangent plane. Assume that $y$ has coordinates given as in (\ref{3}) and $T$ is parametrized as in (\ref{5}) with parameters $p,q,r$. Then the conjugate directions at $y$ with respect to the second fundamental form  correspond to the tangent lines through $y$ to the conic $\Gamma(y)$ introduced in Theorem \ref{quadrics}.
\end{cor}

Given a $\Phi$ surface $X$, the $2$-dimensional families of conjugate directions in its points can also be characterized in terms of the differential equation of second order satisfied by a local parametrization $\phi(t_1, t_2)$ of $X$ (see \S 1). This is discussed in the classical articles (\cite{S},\cite{T1}, \cite{T2}). The integral curves of these two families of tangent directions were classically called \emph{characteristic lines} or \emph{conjugate lines} of the surface. In \cite{S1}, pag. 906, quoted by Dye in \cite{D82} and \cite{D92}, C. Segre wrote: \lq\lq\thinspace {\it Diese Fl\" achen sind diejenigen, die ein doppeltes System von Linien enthalten, analog einem konjugierten doppelten System auf einer Fl\" ache von $S_3$.}'' In my opinion this assertion was misinterpreted by Dye, who thought it meant that the $\Phi$ surfaces must be ruled, while in general the \lq\lq \thinspace characteristic lines'' are not straight lines, nor possibly algebraic curves.

To conclude, I want to mention another interesting classical construction, the Laplace transform of a $\Phi$ surface $X\subset \PP^N$ (see \cite{Da}, \cite{S}). In any point $x\in X$ there is either one pair of conjugate tangent lines or one asymptotic tangent line. So it is possible to consider the  subvarieties $Z_1, Z_2$ of dimension $2$ of the Grassmannian $\Gr(1,N)$ parametrizing the two families of tangent lines. The union of the lines of $Z_i$, $i=1,2$,  results to be a $3$-dimensional subvariety of $\PP^N$ with degenerate Gauss map. One proves  that on any general line of $Z_i$ there are two foci, counted with multiplicity (see \cite{MT}, also for the definition of  foci). One focus is the contact point of the considered tangent line $t$ with $X$, while the other, as $t$ varies in $Z_i$, describes a surface or a curve $X'$, possibly coinciding with $X$, called Laplace transform of $X$.  If $X'$ has dimension $2$, it is again a $\Phi$ surface, and $t$ is a conjugate tangent also for $X'$. So the process can be iterated, giving raise to the Laplace series of $X$. This contruction is performed in \cite{T1} and \cite{T2} in the case of rational $\Phi$ surfaces parametrized by cubics.

I think it would be interesting to determine the characteristic lines and the Laplace transforms of the Kummer surface.



\providecommand{\bysame}{\leavevmode\hbox to3em{\hrulefill}\thinspace}
\providecommand{\MR}{\relax\ifhmode\unskip\space\fi MR }
\providecommand{\MRhref}[2]{%
\href{http://www.ams.org/mathscinet-getitem?mr=#1}{#2}
}
\providecommand{\href}[2]{#2}

\end{document}